\documentclass[reqno]{amsart}
\usepackage[utf8]{inputenc}
\usepackage{amsmath}
\usepackage{amsfonts}
\usepackage{mathrsfs}
\usepackage{amssymb}
\usepackage{amsthm}
\usepackage{enumerate}
\usepackage{appendix}
\usepackage{xcolor}
\usepackage{comment}
\usepackage{hyperref}
\usepackage{cleveref}
\usepackage{geometry}
\usepackage{microtype}
\usepackage{fullpage}
\usepackage{mathtools}

\numberwithin{equation}{section}

\Crefname{equation}{Eq.}{Eqs.}

\newcommand{\BB}{\mathbb{B}}

\newcommand{\dd}{\,\mathrm{d}}
\newcommand{\dom}{\mathcal{D}}
\newcommand{\ee}{\mathrm{e}}
\newcommand{\ii}{\mathrm{i}}
\newcommand{\NN}{\mathbb{N}}

\newcommand{\pd}{\partial}
\newcommand{\ran}{\operatorname{ran}}
\newcommand{\RR}{\mathbb{R}}

\newtheorem{thm}{Theorem}[section]
\newtheorem{cor}[thm]{Corollary}
\newtheorem{lem}[thm]{Lemma}
\newtheorem{prop}[thm]{Proposition}

\theoremstyle{definition}

\newtheorem{defn}[thm]{Definition}

\theoremstyle{remark}
\newtheorem{remark}[thm]{Remark}

\title[Stability of self-similar blow-up solutions to superconformal NLW]{A Note on the Stability of Self-Similar Blow-Up Solutions for Superconformal Semilinear Wave Equations}

\author{Jie Liu}
\address{New York University Abu Dhabi}
\email{jl15817@nyu.edu}

\date{\today}

\begin{document}
	
	\begin{abstract}
		In this note, we investigate the stability of self-similar blow-up solutions for superconformal semilinear wave equations in all dimensions. A central aspect of our analysis is the spectral equivalence of the linearized operators under Lorentz transformations in self-similar variables. This observation serves as a useful tool in proving mode stability and provides insights that may aid the study of self-similar solutions in related problems. As a direct consequence, we establish the asymptotic stability of the ODE blow-up family, extending the classical results of Merle and Zaag [Merle-Zaag, 2007, 2016] to the superconformal case and generalizing the recent findings of Ostermann [Ostermann, 2024] to include the entire ODE blow-up family.
	\end{abstract}
	\maketitle	
	
	\section{Introduction} 
	
	Consider the semilinear wave equation:
	\begin{align}\label{NLW-s1}
		\left\{\begin{array}{l}
			\partial_{tt} u - \Delta u =|u|^{p-1} u, \\
			u(0,x) = u_0(x), \quad \partial_t u(0, x) = u_1(x),
		\end{array}\right.
	\end{align}
	posed in space dimension $N \geq 1$, where the unknown function $u: [0,T) \times \mathbb{R}^N \to \mathbb{R}$, and the initial data satisfy $ (u_0, u_1)\in H_{loc}^k\left(\mathbb{R}^N\right) \times H_{loc}^{k-1}\left(\mathbb{R}^N\right), k>\frac{N}{2}$. The exponent $p>1$ governs the strength of the nonlinearity. In particular, we focus on the superconformal case, i.e.,
	\begin{align*}
		p > p_c \equiv 1+\frac{4}{N-1}, \quad N\ge 2.
	\end{align*}
	Here, $p_c$ is the conformal critical exponent. For the local well-posedness theory of the Cauchy problem \eqref{NLW-s1}, we refer to the classical results of Lindblad and Sogge \cite{lindblad1995existence} and Tao \cite{tao1999low}. The existence of blow-up solutions follows from standard ODE techniques or the energy-based blow-up criterion of Levine \cite{levine1974instability}. 
	
	\subsection{Background}
	The study of blow-up solutions for nonlinear wave equations has attracted significant attention over the past two decades.  For \Cref{NLW-s1}, two fundamental types of blow-up behavior have been identified: 
	\begin{itemize}
		\item 	\textbf{Type-I blow-up}, where the solution diverges at a finite time $T$ at a rate comparable to that of the self-similar ODE solution.
		\item  \textbf{Type-II blow-up}, where the solution remains bounded in a suitable critical norm, allowing for more complex blow-up dynamics.
	\end{itemize}
	
	In the one-dimensional setting ($N=1$), the blow-up theory has been comprehensively developed through a series of works by Merle and Zaag \cite{merle2003determination,merle2007existence,merle2008openness,merle2012existence,merle2012isolatedness}, as well as Côte and Zaag \cite{cote2013construction}. Their results address:
	\begin{itemize}
		\item[-] Blow-up rate \cite{merle2003determination},
		\item[-] Blow-up profiles  \cite{merle2007existence},
		\item[-] Regularity and slope of the blow-up curve \cite{merle2008openness},
		\item[-] Existence of non-characteristic points \cite{merle2012existence,merle2012isolatedness},
		\item[-] Construction of multi-soliton blow-up near non-characteristic points \cite{cote2013construction}.
	\end{itemize}
    A key feature of the one-dimensional case is that all self-similar solutions are obtained as Lorentz transforms of the fundamental ODE blow-up solution, which satisfies
    \begin{align}
    	u_T^*(t) = \frac{\kappa_0}{(T-t)^{\frac{2}{p-1}}},  \quad \kappa_0 = \left(\frac{2(p+1)}{(p-1)^2}\right)^{\frac{1}{p-1}}.
    \end{align}
    
    In higher dimensions ($N \geq 2$), the one-dimensional analogous self-similar solutions persist. More precisely, by exploiting the translation, time shift, and Lorentz invariance of \Cref{NLW-s1}, one can construct a family of self-similar Type-I blow-up solutions depending on $2N+2$ parameters $(T, x_0, d)$, given explicitly by
    \begin{align*}
    	u_{T,x_0,d}^*(t,x) = \kappa_0 \frac{\left(1-|d|^2\right)^{\frac{1}{p-1}}}{(T-t + d \cdot (x-x_0))^{\frac{2}{p-1}}},  \quad T>0,  \quad x_0\in \mathbb{R}^N, \quad |d|<1.
    \end{align*}
    These solutions are well defined within the backward light cone 
    $$
    \Omega^{1,N}(T,x_0) = \left\{(t,x) \in \mathbb{R}^{1,N} \mid |x-x_0|<T-t \right\}.
    $$    
    However, the program of dimension 1 breaks down. In contrast to the one-dimensional case, the classification of self-similar solutions in higher dimensions remains incomplete, except in the radial case outside the origin \cite{merle2011blow}. Indeed, it has been shown that there exist other smooth, radially symmetric self-similar solutions to \Cref{NLW-s1}, as demonstrated by P. Bizo\'n et al. \cite{bizon2010self, bizon2007self} in three dimensions for odd powers, and by W. Dai and T. Duyckaerts \cite{dai2021self} for $N\ge 3$ and energy-supercritical powers $1+\frac{4}{N-2}<p<1+\frac{4}{N-3}$, see also the recent works of Csobo, Glogić, and Schörkhuber \cite{csobo2024blowup}; Chen, Donninger, Glogić, McNulty, and Schörkhuber \cite{chen2024co}; and Kim \cite{kim2022self}.

	Nevertheless, partial results on blow-up theory are available. For subconformal nonlinearities ($p \leq p_c$), the blow-up rate has been established \cite{merle2003determination,merle2005determination}, and the stability of self-similar solutions $u_{T,x_0, d}^*$ has been proven by Merle and Zaag \cite{merle2015stability,merle2016dynamics}. As a corollary, they showed that the set 
	\begin{align*}
		\mathscr{R}_0=\left\{x_0 \in \mathscr{R}\,|\,\exists \, |d(x_0)| <1,\exists \; e(x_0)=\pm 1 \text { s.t. } \begin{pmatrix}
			U_{x_0}(s) \\
			\pd_s U_{x_0}(s)
		\end{pmatrix}\rightarrow e(x_0)  \begin{pmatrix}
		\kappa(d(x_0)) \\
		0
		\end{pmatrix}  \text{ in } \mathscr{H} \text { as } s \rightarrow \infty\right\}
	\end{align*}
	is open, where  $\mathscr{R}$ is the set of characteristic points, $U_{x_0}:= (T-t)^{\frac{2}{p-1}} u(t,x)$ is the blow-up profile of $u$ in self-similar variables, $\kappa(d)$ is the corresponding blow-up profile of $u^{*}_{T,x_0,d}$, and $\mathscr{H}$ is the energy space. Moreover, they obtained the differentiability of the blow-up surface $x \mapsto T(x)$. A different approach to stability is also presented in \cite{donninger2012stable}. 
	
	For superconformal exponents ($p > p_c$), significantly fewer results are available. Donninger et al. \cite{donninger2014stable, donninger2016blowup, donninger2017stable, donninger2017strichartz, donninger2020blowup, wallauch2023strichartz, ostermann2024stable} have used spectral-theoretic methods to establish the stability of the fundamental ODE blow-up solution $u_T^*$. These techniques have been extended to other self-similar solutions in supercritical regimes \cite{glogic2020threshold, glogic2021co, csobo2024blowup, chen2024co}. Recently, Kim \cite{kim2022self} proved the co-dimensional stability of self-similar profiles using a functional framework introduced in \cite{MRR22}.
	
	For energy-critical and supercritical nonlinear wave equations, another type of blow-up solution, known as \textit{Type-II blow-up}, arises. Comprehensive surveys on Type-II blow-up can be found in \cite{krieger2009slow,krieger2014full,krieger2020stability,hillairet2012smooth,collot2018type,duyckaerts2012universality,ghoul2018construction} and the references therein.
	
	
	Despite these advances, the stability of the full family of self-similar solutions (including those obtained via Lorentz transformations) remains an open problem. Notably, the stability of the fundamental ODE blow-up solution does not directly imply the stability of its Lorentz-boosted counterparts, see our discussion in Section \ref{subsection-Lorentz}.

	\subsection{Main result}
	
	In this paper, we bridge this gap by establishing the spectral stability of the full family of self-similar solutions obtained via Lorentz transformations in the superconformal regime. Our approach leverages the Lorentz transformation in self-similar variables, extending the stability results of Merle and Zaag \cite{merle2016dynamics} to the superconformal setting. More precisely, we aim to prove the stability of general self-similar solutions $u_{T,x_0,d}^*$ for all $|d|<1$. 

	To state our results, we introduce the following similarity variables, for any $T>0$ and $x_0 \in \RR^N$:
	\begin{align*}
		u(t, x) = \frac{1}{\left(T-t\right)^{\frac{2}{p-1}}}  U_{T,x_0}(s, y), \quad s=-\log \left(T-t\right) + \log T, \quad y=\frac{x-x_0}{T-t}.
	\end{align*}
	The function $U_{T,x_0}$ (we write $U$ for simplicity)  satisfies the following equation:
	\begin{align}\label{nlw-ss-s2}
		\partial_{ss} U + \frac{p+3}{p-1} \partial_s U + 2 y\cdot\nabla \partial_s U +   \frac{2(p+1)}{p-1} y\cdot\nabla U + (y_i y_j - \delta_{ij}) \partial_i \partial_j U +\frac{2(p+1)}{(p-1)^2} U = |U|^{p-1} U
	\end{align}
	for all $|y|<1$ and $s \geq 0$. Stationary solutions to \Cref{nlw-ss-s2} are called self-similar solutions. For all $ d\in \BB^N$, the corresponding  self-similar profiles of $u^*_{T,x_0,d}$ are given by 
	\begin{align*}
		U^*_{T,x_0,d}(s,y) := \left(T-t\right)^{\frac{2}{p-1}} u^*_{T,x_0,d}(t,x) \equiv 	\kappa_d(y),
	\end{align*}
	where
	\begin{align*}
		\kappa_d(y)=\kappa_0 \frac{\left(1-|d|^2\right)^{\frac{1}{p-1}}}{(1+d \cdot y)^{\frac{2}{p-1}}},  \quad \kappa_0=\left(\frac{2(p+1)}{(p-1)^2}\right)^{\frac{1}{p-1}}, \quad \ |y|\le1.
	\end{align*}
	In particular, when $d=0$, we obtain the ODE blow-up profile
	\begin{align*}
		\kappa(0, y) \equiv \kappa_0. 
	\end{align*}
	
	The main theorem is as follows.
	\begin{thm}
		\label{MainTHM}
		Let $d_0\in \BB^N$. Let $(N,p,k,R)\in\NN\times\RR_{>1}\times\NN\times\RR_{\ge1} $ with $k>\frac{N}{2}$ and $R<|d_0|^{-1}$. Put
		\begin{equation*}
			\omega_{p} = \min\{1,s_{p}\}
			\qquad\text{where}\qquad
			s_{p} = \frac{2}{p-1} \,.
		\end{equation*}
	   For all $ 0 < \varepsilon < \omega_{p}$ there are constants $0<\delta_{\varepsilon}<1-|d_0|$ and $C_{\varepsilon} > 1$ such that for all $0<\delta\leq\delta_{\varepsilon}$, $C\geq C_{\varepsilon}$ and all real-valued $(f,g)\in C^{\infty}(\RR^{N})\times C^{\infty}(\RR^{N})$ with
		\begin{equation*}
			\| (f,g) \|_{H^{k}(\RR^{N})\times H^{k-1}(\RR^{N})} \leq \frac{\delta}{C}
		\end{equation*}
		there exist parameters $d^{*}\in\overline{\BB^{N}_{\delta}(d_0)}\subset \BB^{N}$ and $T^{*}\in \overline{\BB^{1}_{\delta}(1)}$ and a unique solution $u\in C^{\infty}\big(\Omega^{1,N}_R(T^{*})\big)$ in the past light cone
		\begin{equation*}
			\Omega^{1,N}_R(T^{*}) = \left\{ (t,x) \in \RR^{1,N} \mid 0\leq t < T^{*}, \, |x| \leq R(T^{*} - t) \right\}
		\end{equation*}
		to the Cauchy problem
		\begin{equation}
			\label{CauchyProblem}
			\renewcommand{\arraystretch}{1.2}
			\left\{
			\begin{array}{rcll}
				\pd_{tt}u(t,x) - \Delta u(t,x)&=&|u(t,x)|^{p-1}u(t,x) \,, & \quad (t,x)\in \Omega_R^{1,N}(T^{*}) \,, \\
			    u(0,x) &=&u_{1,0,d_0}^{*}(0,x) + f(x) \,, & \quad x\in\BB^{N}_{RT^*} \,, \\
				(\pd_{t}u)(0,x) &=& (\pd_{t}u_{1,0,d_0}^{*})(0,x) + g(x) \,, & \quad x\in\BB^{N}_{RT^*} \,,
			\end{array}
			\right.
		\end{equation}
		such that the bounds
		\begin{align*}
			(T^{*}-t)^{-\frac{N}{2}+s_{p}+s}
			\left\| u(t,\,.\,) - u_{T^{*}, 0, d^{*}}^{*}(t,\,.\,) \right\|_{\dot{H}^{s}(\BB^{N}_{R(T^{*}-t)})}
			&\lesssim
			(T^{*}-t)^{\omega_{p} - \varepsilon}
			\intertext{for $s=0,1,\ldots,k$, and}
			(T^{*}-t)^{-\frac{N}{2}+s_{p}+s}
			\left\| \pd_{t}u(t,\,.\,) - \pd_{t}u_{T^{*},0, d^{*}}^{*}(t,\,.\,) \right\|_{\dot{H}^{s-1}(\BB^{N}_{R(T^{*}-t)})}
			&\lesssim
			(T^{*}-t)^{\omega_{p} - \varepsilon}
			\intertext{for $s=1,\ldots,k$, hold for all $0\leq t < T^{*}$.}
		\end{align*}
	\end{thm}
	
\begin{remark}
	Due to the invariance under space and time translations, \Cref{MainTHM} also applies to the general solution $u_{T_0,x_0,d_0}^{*}$ for any $x_0 \in \mathbb{R}$ and $T_0 > 0$. Without loss of generality, we henceforth fix $x_0 = 0$ and $T_0 = 1$.
\end{remark}
	
	As a corollary, we obtain the following dynamics of equation \eqref{nlw-ss-s2} for initial data near $ \kappa_d(y)$  in self-similar variables, which extends the main result of \cite{merle2016dynamics} to the superconformal case.
	\begin{thm}[Behavior of solutions of equation \eqref{nlw-ss-s2} near $ \kappa(d, y)$] 	
		\label{thm-trapping}
		Let $d_0\in \BB^N$. Let $(N,p,k)\in\NN\times\RR_{>1}\times\NN$ with $k>\frac{N}{2}$. 
		There exist $\delta>0, C>0$ and $\varepsilon>0$ such that for any smooth solution $U $ of equation \eqref{nlw-ss-s2} satisfying that $(U,\pd_s U)$ are continuous in time with values in ${H}^k(\BB^N)\times {H}^{k-1}(\BB^N)$,  if
		\begin{align*}
		 \left\|\binom{U(0)}{\partial_s U(0)}-\binom{\kappa_{d_0}(y)}{0}\right\|_{{H}^k(\BB^N)\times {H}^{k-1}(\BB^N)} \leq \delta
		\end{align*}		
	then:
	\begin{itemize}
		\item[-] either $U(s)$ is not defined for all $(s,y) \in [0,+\infty)\times \BB^N$, 
		\item[-]  or		
		\begin{align*}
			\left\| \binom{U(s)}{\partial_s U(s)} \right\|_{{H}^k(\BB^N)\times {H}^{k-1}(\BB^N)} \rightarrow 0 \text { as } s \rightarrow \infty \text { exponentially fast }
		\end{align*}
		\item[-] or there exists $d^{*}\in\overline{\BB^{N}_{\delta}(d_0)}\subset \BB^{N}$  such that		
		\begin{align*}
			\forall s \geq 0, \  \left\|\binom{U(s)}{\partial_s U(s)}- \binom{\kappa_{d^*}(y)}{0} \right\|_{{H}^k(\BB^N)\times {H}^{k-1}(\BB^N)} \leq C \delta  e^{-\varepsilon s}.
		\end{align*}			 
	\end{itemize}
	\end{thm}
\begin{remark}
	Our result remains valid in the subconformal and conformal cases, providing an alternative proof of the previously established results in \cite{merle2007existence,merle2016dynamics}.
\end{remark}
  

\subsection{Idea of the proof} 

As pointed out in \cite{merle2016dynamics}, the energy method based on the existence of a Lyapunov functional fails in the superconformal case. In fact, in our setting, the corresponding Lyapunov functional does not even exhibit dissipation. To overcome this difficulty, Donninger and collaborators \cite{donninger2012stable,donninger2014stable,donninger2016blowup,donninger2017stable,donninger2017strichartz,donninger2020blowup,donninger2024spectral} developed a spectral method, successfully deriving the asymptotic stability of the ODE blow-up solution, i.e., $u^*_{T}$. This spectral approach is particularly powerful as it is well suited to the non-selfadjoint nature of the linearized problem, especially in the presence of unstable modes.  See also our recent extension \cite{ghoul2025blow}, which generalizes this method to accommodate non-compact perturbations.

Originally introduced by Donninger in the context of self-similar blow-up solutions for wave maps \cite{donninger2011stable,donninger2012wavemap,donninger2023optimal}, this method has found applications in various settings \cite{costin2016proof,costin2016stability,costin2017mode,chen2023singularity,chen2024co,mcnulty2024singularity}. A key step in this approach is to establish the mode stability of the linearized operator around the self-similar solution, which requires proving that the only non-negative eigenvalues of the linearized operator are $0$ and $1$. The mode stability problem reduces to solving the eigenvalue equation, which takes the form of a second-order elliptic PDE with singular coefficients:
\begin{align}\label{eigen-eq-intro}
	\left(\lambda^2+\frac{p+3}{p-1}\lambda + \frac{2(p+1)}{(p-1)^2} - p|\kappa_d|^{p-1}  \right)\varphi + \left(2\lambda+\frac{2(p+1)}{p-1} \right)y\cdot \nabla \varphi  +(y_iy_j - \delta_{ij})\partial_{i}\partial_j \varphi =0,
\end{align}
where
$$
|\kappa_d|^{p-1}  = \kappa_0^{p-1} \frac{\left(1-|d|^2\right)}{(1+d \cdot y)^{2}},
$$
see \Cref{defn-eigen}.

For the special case $d=0$, the equation \eqref{eigen-eq-intro} can be decomposed in spherical coordinates, reducing the problem to a family of one-dimensional hypergeometric ODEs, which admit explicit solutions (see \cite{ostermann2024stable}). However, when $d\neq 0$, this decomposition fails due to the presence of an additional singularity introduced by $\kappa_d$, namely the set 
$$
\{ y \in \mathbb{R}^N \mid d \cdot y = -1 \}.
$$
Even in one dimension, the resulting equation becomes a Heun-type ODE, which is known to be analytically intractable in general \cite{donninger2024spectral}.

To circumvent this difficulty, we employ a Lorentz transformation in self-similar variables. This transformation was previously introduced by Merle and Zaag in the one-dimensional case \cite{merle2007existence} and later extended to higher dimensions \cite{merle2016dynamics}, where it was used to establish the coercivity of the energy around $\kappa_d$ in the subconformal setting. Our key observation is that this transformation allows us to establish full spectral equivalence of the linearized operators for all $d \in \mathbb{B}^N$. Consequently, the mode stability problem for $\kappa_d$ can always be reduced to the mode stability of the ODE blow-up solution $\kappa_0$, which has already been established in \cite{ostermann2024stable}. As a direct corollary, we obtain the asymptotic stability of $\kappa_d$.

The remainder of the paper is organized as follows. In \Cref{sec:mode}, we establish mode stability and prove the equivalence of the discrete spectrum under Lorentz transformations. In \Cref{sec:linear}, we extend this equivalence to the full spectrum and derive the linear stability of $\kappa_d$. In \Cref{sec:nonlinear}, we prove the main result on nonlinear stability. Since the nonlinear analysis largely follows the strategy developed in \cite{ostermann2024stable} for the ODE blow-up $\kappa_0$, we defer it to the appendix for completeness, emphasizing only the necessary adaptations to extend the result to the entire ODE blow-up family.

\subsection{Notation}
For consistency, we adopt the notation from \cite{ostermann2024stable}. The Einstein summation convention is assumed throughout. The sets of natural and real numbers are denoted by $\mathbb{N}$ and $\mathbb{R}$, respectively. In the $N$-dimensional Euclidean space $\mathbb{R}^N$, the open ball of radius $R>0$ centered at the origin is denoted by $\mathbb{B}_R^N \subset \mathbb{R}^N$, while the corresponding closed ball is given by $\overline{\mathbb{B}_R^N} \subset \mathbb{R}^N$.

For $w\in \mathbb{R}$, we denote $\mathbb{H}_{w}=\{z\in \mathbb{C} \mid \operatorname{Re} z >w\}$ for a right half-plane and $\overline{\mathbb{H}_{w}} = \{z\in \mathbb{C} \mid \operatorname{Re} z \ge w\}$ a closed right half-plane.

We use boldface notation for tuples of functions, for example:
\begin{align*}
	\mathbf{f} \equiv\left(f_1, f_2\right) \equiv\left[\begin{array}{l}
		f_1 \\
		f_2
	\end{array}\right] \quad \text { or } \quad \mathbf{q}(t, .) \equiv\left(q_1(t, .), q_2(t, .)\right) \equiv\left[\begin{array}{l}
		q_1(t, .) \\
		q_2(t, .)
	\end{array}\right].
\end{align*}
Linear operators that act on tuples of functions are also displayed in boldface notation. For a closed linear operator $\mathbf{L}$ on a Banach space, we denote its domain by $\mathcal{D}(\mathbf{L})$, its spectrum by $\sigma(\mathbf{L})$, and its point spectrum by $\sigma_p(\mathbf{L})$. The resolvent operator is denoted by $\mathbf{R}_{\mathbf{L}}(z):=(z-\mathbf{L})^{-1}$ for $z \in \rho(\mathbf{L})=\mathbb{C} \backslash \sigma(\mathbf{L})$. The space of bounded operators on a Banach space $\mathcal{X}$ is denoted by $\mathcal{L}(\mathcal{X})$.

For details on the spectral theory of linear operators, we refer to \cite{kato1995}. The theory of strongly continuous operator semigroups is covered in the textbook \cite{engel2000one}.
	

	\section{Equivalence of Discrete Spectra and Mode Stability}\label{sec:mode}
	
	In this section, we introduce the Lorentz transformation in self-similar variables and establish the equivalence of the discrete spectra of the linearized operator associated with equation \eqref{nlw-ss-s2} around each stationary solution $\kappa_d$, where $d \in \mathbb{B}^N$. As a consequence, we derive the mode stability of these linearized operators.
	
	Consider then $U(s, y)$ a solution to equation \eqref{nlw-ss-s2}. For $|d|<1$, let $U  =  \kappa_d + \eta$, then the corresponding linearized equation at $\kappa_d $ is given by 
		\begin{align}\label{eq:linearized-ss}
			\begin{aligned}
		&\partial_{ss} \eta + \frac{p+3}{p-1} \partial_s \eta + 2 y\cdot\nabla \partial_s \eta +   \frac{2(p+1)}{p-1} y\cdot\nabla \eta + (y_i y_j - \delta_{ij}) \partial_i \partial_j \eta +\left(\frac{2(p+1)}{(p-1)^2} -  p|\kappa_d|^{p-1} \right) \eta = 0.
	\end{aligned}
	\end{align}
	Denote 
	$$ \mathbf{q} = \begin{pmatrix}
		q_1 \\ q_2
	\end{pmatrix}
	 =  \begin{pmatrix}
	 	\eta \\ \partial_s \eta + y\cdot\nabla  \eta + \frac{2}{p-1} \eta
	 \end{pmatrix}, $$
	then Eq. \eqref{eq:linearized-ss} can be rewritten as the following first-order system
	\begin{align*}
		\partial_s \begin{pmatrix}
			q_1 \\ q_2
		\end{pmatrix} & = \begin{pmatrix}
		-y\cdot \nabla q_1 - \frac{2}{p-1}q_1 + q_2\\
		\Delta q_1 + p|\kappa_d|^{p-1} q_1  -y\cdot \nabla q_2 - \frac{p+1}{p-1}q_2 
		\end{pmatrix} =:  \tilde{\mathbf{L}}_d  \begin{pmatrix}
		q_1 \\ q_2
		\end{pmatrix}
		\end{align*}


\subsection{The Lorentz transformation in self-similar variables}\label{subsection-Lorentz}
 We recall the N-dimensional Lorentz transformation $\mathcal{T}_\beta$ and its inverse $\mathcal{T}^{-1}_\beta = \mathcal{T}_{-\beta}$ such that $(t',x') = \mathcal{T}_\beta (t,x)$ are defined as follows
\begin{align}
	\left\{ \begin{aligned}
		&t' = \gamma (t-\beta\cdot x), \\ 
		&x'_i = - \gamma \beta_i t + x_i + \frac{\gamma^2\beta_i}{1+\gamma} (\beta\cdot x),
	\end{aligned}\right. \quad \text{~and~} \left\{ \begin{aligned}
		&t = \gamma (t'+\beta\cdot x'), \\ 
		&x_i =  \gamma \beta_i t' + x'_i + \frac{\gamma^2\beta_i}{1+\gamma} (\beta\cdot x'),
	\end{aligned}\right.
\end{align}
where $\gamma = (1-|\beta|^2)^{-\frac12}$ and $|\beta|<1$. Let $v(t',x') = u(t,x) = u(t(t',x'),x(t',x'))$, then by the Lorentz transformation invariance of \eqref{NLW-s1}, if $u$ solves \eqref{NLW-s1}, $v(t',x')$ satisfies the same equation
\begin{align}\label{nlw-transform}
	v_{t't'} - \Delta v = |v|^{p-1}v.
\end{align}
And under this transformation, the self-similar blow-up solution 
\begin{align*}
	u^*_{T,x_0, d}(t,x) &= \kappa_0 \frac{\left(1-|d|^2\right)^{\frac{1}{p-1}}}{(T-t + d \cdot (x-x_0))^{\frac{2}{p-1}}}, 
\end{align*}
corresponds to 
\begin{align*}
	v^*_{T,x_0, d, \beta}(t',x') &:=  \kappa_0 \frac{\left(1-|d|^2\right)^{\frac{1}{p-1}}}{\left(T-d\cdot x_0 + \gamma(d\cdot \beta -1) t' + \left( \frac{\gamma^2 \beta\cdot(d-\beta)}{1+\gamma}\beta + (d-\beta) \right) \cdot x' \right)^{\frac{2}{p-1}}}. 
\end{align*}
In particular, if we take $\beta = d$, then $v^*_{T,x_0, d, \beta}$ is independent of the spatial variable $x'$, i.e 
$$v^*_{T,x_0,d,d}(t',x')  =  \frac{ \kappa_0}{\left(\frac{T-d\cdot x_0}{\sqrt{1-|d|^2}}-t'\right)^{\frac{2}{p-1}}} =   \frac{ \kappa_0}{(T'-t')^{\frac{2}{p-1}}}\equiv v^*_{T'}, \quad T'= \frac{T-d\cdot x_0}{\sqrt{1-|d|^2}}, $$ 
which is the ODE blow-up solution to Eq. \eqref{nlw-transform}. Note that even with this transformation it is not easy to directly reduce the stability of $u^*_{T,x_0, d}$ to the stability of the ODE blow-up $v^*_{T'}$ due to the following reasons:
\begin{enumerate}
	\item The initial perturbations to $u_{T,x_0,d}^*$ and $v^*_{T'}$ are imposed on different regions, i.e. $\{(t,x) \in \mathbb{R}^2 : t=0\}$ and  $\{(t',x')  \in \mathbb{R}^2 : t'=0\} = \{ (t,x) \in \mathbb{R}^2 : t- d\cdot x=0\}$ respectively.
	\item The stable evolution in $t'$ direction does not necessarily imply a stable evolution in $t$ direction.
\end{enumerate}
Nevertheless, we shall see that it allows us to study the spectrum of the linearized operator at $u_{T,x_0,d}^*$ by analyzing the spectrum of the linearized operator at $v^*_{T'}$ in self-similar variables, which facilitates a rigorous proof of the mode stability of $\kappa_d$. To begin with, we consider the Lorentz transformation in self-similar variables which was first introduced by \cite{merle2007existence} for one-dimensional NLW with power nonlinearities. 

For $T'>0, x_0'\in \mathbb{R}$, define the corresponding self-similar variables for $t'$ and $x'$ as 
$$ s' = -\log(T'-t') + \log(T'), \quad y' =\frac{x'-x_0'}{T'-t'},$$
and let $V(s',y') = (T'- t')^{\frac{2}{p-1}} v(t',x')$, if $v(t',x')$ solves Eq. \eqref{nlw-transform}, then $V(s',y')$ satisfies \eqref{nlw-ss-s2} as well, i.e.,
\begin{align}\label{nlw-ss-transform}
	\begin{aligned}
		\partial_{s's'} V + \frac{p+3}{p-1} \partial_{s'} V + 2 y'\cdot\nabla \partial_{s'} V +   \frac{2(p+1)}{p-1} y'\cdot\nabla V + (y'_i y'_j - \delta_{ij}) \partial_i \partial_j V +\frac{2(p+1)}{(p-1)^2} V = |V|^{p-1} V
	\end{aligned}
\end{align}
A direct computation implies the following transformation in self-similar variables, see also \cite[Lemma A.1.]{merle2016dynamics}
\begin{lem}[The Lorentz transform in self-similar variables]\label{prop-lorentz-ss} Consider $U(s,y)$ a solution to \eqref{nlw-ss-s2} defined for all $|y|<1$, and introduce for any $|\beta|<1$, the function $V \equiv \mathcal{T}_{\beta} U$ defined by
	\begin{align*}
		V(s',y') = \left(\frac{1}{\gamma(1-\beta\cdot y')}\right)^{\frac{2}{p-1}} U(s,y), \text{~where~}  s =s' - \log(1-\beta\cdot y'), \quad y=\frac{y'-\gamma\beta + \frac{\gamma^2 \beta\cdot y'}{1+\gamma}\beta}{\gamma(1-\beta \cdot y')},
	\end{align*}
	then $V(s',y')$ solves \eqref{nlw-ss-transform} for all $|y'|<1$.
\end{lem}

\subsection{Unstable eigenvalues induced by symmetries}
	We postpone the rigorous functional setup of the linearized operator $\tilde{\mathbf{L}}_{d}$ to the next section, and regard $\tilde{\mathbf{L}}_{d}$ as a formal differential operator. We first study the possible unstable eigenvalues of $\tilde{\mathbf{L}}_{d}$ as a step towards proving its mode stability. 
	\begin{defn}\label{defn-eigen}
		$\lambda \in \mathbb{C}$ is called an eigenvalue of $\tilde{\mathbf{L}}_{d}$ if there exists a nontrivial smooth function $\varphi\in C^\infty(\overline{\mathbb{B}})$ such that 
		\begin{align}\label{eigen-eq}
			\left(\lambda^2+\frac{p+3}{p-1}\lambda + \frac{2(p+1)}{(p-1)^2} - p|\kappa_d|^{p-1}  \right)\varphi + \left(2\lambda+\frac{2(p+1)}{p-1} \right)y\cdot \nabla \varphi  +(y_iy_j - \delta_{ij})\partial_{i}\partial_j \varphi =0. 
		\end{align}
		We call an eigenvalue $\lambda$  unstable if its real part is non-negative, i.e., $\operatorname{Re} \lambda \ge 0$. Otherwise, we call $\lambda$ a stable eigenvalue.
	\end{defn}
	\begin{remark}
		Actually, if $\tilde{\mathbf{L}}_{d} \mathbf{q} = \lambda \mathbf{q}$ for some $\mathbf{q} = (q_1,q_2)^T \in (C^\infty(\overline{\mathbb{B}}))^2$, then $q_1$ satisfies \eqref{eigen-eq} and $\lambda $ is an eigenvalue of $\tilde{\mathbf{L}}_{d}$. Conversely, if $\lambda $ is an eigenvalue of $\tilde{\mathbf{L}}_{d}$ with a smooth eigenfunction $\varphi(y)$, then $\mathbf{q} := (\varphi, \lambda\varphi + y\cdot \nabla \varphi + \frac{2}{p-1}\varphi)^T$ satisfies $\tilde{\mathbf{L}}_{d} \mathbf{q} = \lambda \mathbf{q}$.
	\end{remark}
	
		
	First, the symmetries of Eq. \eqref{nlw-ss-s2} give us some explicit unstable eigenvalues. More precisely, if $U(s,y)$ is a solution, then the following transformations of $U(s,y)$ are also solutions to \eqref{nlw-ss-s2}:
	\begin{itemize}
		\item (Space translation in $x$) for any $a \in \mathbb{R}^N$, the function
		\begin{align*}
			U_1(s, y)= U\left(s, y+a e^s\right);
		\end{align*}
		\item (Time translation in $t$) for any $b \in \mathbb{R}$, the function 
		\begin{align*}
			U_2(s, y)= \left(1+b e^s\right)^{-\frac{2}{p-1}}U\left(s-\log \left(1+b e^s\right), \frac{y}{1+b e^s}\right);
		\end{align*} 
		\item(Lorentz transformation) The transposition in self-similar variables of the Lorentz transformation 
		\begin{align*}
			U_3(s, y)= \mathcal{T}_{\beta}U\left(s, y\right);
		\end{align*}
	\end{itemize}
    Thus, for $U = \kappa_d(y)$, we have
	\begin{itemize}
		\item for all $1\le i\le N$, $\partial_{a_i} U_1|_{a=0}$ and $ \partial_{b} U_2|_{b=0}$ generate  the same 1d eigenspace of eigenvalue 1; 
		\item  for $1\le i\le N$, $\partial_{\beta_i} U_3|_{\beta=0} $ generates a N dimensional eigenspace of eigenvalue 0.
	\end{itemize}
 Therefore, the linearized operator $\tilde{\mathbf{L}}_{d}$ has at least two unstable eigenvalues $0$ and $1$. More precisely, we have 
 
 	\begin{lem}[{\cite[Lemma 3.2]{ostermann2024stable}}]
 	\label{EV}
 	Let $(N,p,R)\in \NN\times\RR_{>1}\times\RR_{\ge 1}$, $d\in \BB^N$, and denote $s_p := \frac{2}{p-2}$. Let $\mathbf{f}_{0,d,i},\mathbf{f}_{1,d}\in C^{\infty}(\overline{\BB^{N}_R})^{2}$ be given by
 	\begin{align}
 		\label{EV0beta}
 		\mathbf{f}_{0,d,i}(y) &=
 		\begin{bmatrix}
 			\hfill s_{p} \kappa_0 \gamma(d)^{-s_{p}} (1+d\cdot y)^{-s_{p}-1} y_{i} \\
 			s_{p} (s_{p}+1) \kappa_0 \gamma(d)^{-s_{p}} (1+d\cdot y)^{-s_{p}-2} y_{i}
 		\end{bmatrix}
 		\\&\nonumber\,+
 		\begin{bmatrix}
 			\hfill s_{p} \kappa_0 \gamma(d)^{-s_{p}+2} (1+d\cdot y)^{-s_{p}} d_{i} \\
 			s_{p}^{2} \kappa_0 \gamma(d)^{-s_{p}} (1+d\cdot y)^{-s_{p}-1} d_{i}
 		\end{bmatrix}
 		\,, \\
 		\label{EV1beta}
 		\mathbf{f}_{1,d}(y) &=
 		\begin{bmatrix}
 			\hfill s_{p} \kappa_0 \gamma(d)^{-s_{p}} (1+d\cdot y)^{-s_{p}-1} \\
 			s_{p} (s_{p}+1) \kappa_0 \gamma(d)^{-s_{p}} (1+d\cdot y)^{-s_{p}-2}
 		\end{bmatrix}
 		\,,
 	\end{align}
 	for $i=1,\ldots,N$. Then
 	\begin{equation*}
 		\tilde{\mathbf{L}}_{d} \mathbf{f}_{0,d,i} = \mathbf{0}
 		\qquad\text{and}\qquad
 		(\mathbf{I}-\tilde{\mathbf{L}}_{d}) \mathbf{f}_{1,d} = \mathbf{0} \,.
 	\end{equation*}
 \end{lem}

\subsection{Mode stability}\label{subsec:mode-stability}
   Nevertheless, we introduce the following notion of mode stability, see \cite{donninger2011stable,donninger2024spectral}. 
   \begin{defn}[Mode stability]
   	We say that the blow-up solution $\kappa_d$ is mode stable if the existence of a nontrivial smooth function $\varphi \in C^{\infty}(\overline{\mathbb{B}})$ that satisfies Eq. \eqref{eigen-eq} necessarily implies that $\lambda \in \{0,1\}$ or $\operatorname{Re} \lambda<0$.
   \end{defn}
   
   To study the mode stability of $\kappa_d$, we need to solve Eq. \eqref{eigen-eq}. Note that $\kappa_d(y)$ is non-radial, thus we are not able to reduce \eqref{eigen-eq} to one-dimensional ODEs by using spherical harmonics decomposition, see \cite{donninger2016blowup,ostermann2024stable}. Moreover, $\kappa_d(y)$ has singular points $\{y\in \mathbb{R}^N: d\cdot y =-1\}$, which makes the ODE analysis of \eqref{eigen-eq} complicated. 	
	
   To deal with this problem, we note that  $\kappa_d$ can be transformed to the ODE blow-up profile using the Lorentz transformation. More precisely,   
	 for $\kappa_d(y) =\kappa_0 \frac{\left(1-|d|^2\right)^{\frac{1}{p-1}}}{(1+d \cdot y)^{\frac{2}{p-1}}}$, 
	the corresponding function $V_{d,\beta}(y') \equiv \mathcal{T}_{\beta} \kappa(d,\cdot)$ solves Eq. \eqref{nlw-ss-transform} and satisfies
	\begin{align*}
		V_{d,\beta}(y') &=  \left(\frac{1}{\gamma(1-\beta\cdot y')}\right)^{\frac{2}{p-1}}  \frac{ \kappa_0 \left(1-|d|^2\right)^{\frac{1}{p-1}}}{(1+d \cdot y)^{\frac{2}{p-1}}}\\
		& =  \left(\frac{1}{\gamma(1-\beta\cdot y')}\right)^{\frac{2}{p-1}}  \frac{\kappa_0 \left(1-|d|^2\right)^{\frac{1}{p-1}}}{\left(1+\beta \cdot \frac{y'-\gamma\beta + \frac{\gamma^2 \beta\cdot y'}{1+\gamma}\beta}{\gamma(1-\beta \cdot y')} + (d-\beta)\cdot y \right)^{\frac{2}{p-1}}}\\
		&  =  \left(\frac{1}{\gamma(1-\beta\cdot y')}\right)^{\frac{2}{p-1}}   \frac{\kappa_0  \left(1-|d|^2\right)^{\frac{1}{p-1}}}{\left( \frac{1}{\gamma^2 (1-\beta\cdot y')}+ (d-\beta)\cdot y \right)^{\frac{2}{p-1}}}\\
		& =  \frac{\kappa_0  \left( \gamma^2 (1-|d|^2)\right)^{\frac{1}{p-1}}}{\left( 1 + \gamma^2 (1-\beta\cdot y')(d-\beta)\cdot y \right)^{\frac{2}{p-1}}}
	\end{align*}
	Again, when $\beta =d$, $V_{d,d} = \kappa_0(y) \equiv \kappa_0 $ is the ODE blow-up profile. Hereafter, we always take $\beta = d$. Let $V= V_{d,d} + \xi = \kappa_0 + \xi$, we derive the linearized equation of  Eq. \eqref{nlw-ss-transform}  at $V_{d,d} $ 
	\begin{align}\label{linearized-ss-transform}
		\partial_{s's'} \xi +  \frac{p+3}{p-1} \partial_{s'} \xi + 2 y'\cdot\nabla \partial_{s'} \xi +   \frac{2(p+1)}{p-1} y'\cdot\nabla \xi + (y'_i y'_j - \delta_{ij}) \partial_i \partial_j \xi +\left(\frac{2(p+1)}{(p-1)^2} -  p|\kappa_0|^{p-1} \right) \xi  = 0.
	\end{align}	
	Denote $\mathbf{r} = (r_1, r_2) = (\xi, \partial_{s'} \xi + y'\partial_{y'} \xi + \frac{2}{p-1}\xi)$, then we can rewrite \eqref{linearized-ss-transform} as a first-order PDE system
	\begin{align}\label{linear-ss-system-transform}
		\partial_s \left(\begin{array}{l}
			r_1  \\
			r_2  
		\end{array}\right) &= \left(\begin{array}{c}
			- y'\cdot \nabla r_1 -\frac{2}{p-1} r_1 + r_2 \\
			\Delta r_1 + p|\kappa_0|^{p-1} r_1  -y'\cdot \nabla r_2 - \frac{p+1}{p-1}r_2 
		\end{array}\right)  = \tilde{\mathbf{L}}_0  \left(\begin{array}{l}
			r_1  \\
			r_2  
		\end{array}\right).
	\end{align}
	
	Following \Cref{prop-lorentz-ss}, we show that the Lorentz transformation $\mathcal{T}_d$ also transforms the linearized equation around $\kappa_d$ to the linearized equation around $\kappa_0$.
	\begin{prop}[Transformation of the linearized equation]\label{prop-transform-linear-op} Consider $\eta(s,y)$ a solution to the  linearized equation \eqref{eq:linearized-ss} around $\kappa_d(\cdot)$, and introduce $\xi \equiv \mathcal{T}_d \eta$, then $\xi$ satisfies the  linearized equation \eqref{linearized-ss-transform} around $\kappa_0(\cdot)\equiv \kappa_0$.
	\end{prop}
\begin{proof}
	Observe that the full nonlinear wave equation \eqref{nlw-ss-s2} is invariant under the Lorentz transformation $\mathcal{T}_d$, and that $\mathcal{T}_d \kappa_d = \kappa_0$. Moreover, the transformation $\mathcal{T}_d$ maps the nonlinear term $|U|^{p-1} U$ to $\left(\gamma(1 - d \cdot y') \right)^{-2p/(p-1)} |V|^{p-1} V$. Consequently, $\mathcal{T}_d$ maps the linearized equation around $\kappa_d$ to the corresponding equation around $\kappa_0$.
\end{proof}

	\begin{cor}[Equivalence of the discrete spectrum]\label{cor-eigen}
		$\lambda$ is an eigenvalue of $\tilde{\mathbf{L}}_{d}$ if and only if $\lambda$ is an eigenvalue of $\tilde{\mathbf{L}}_{0}$.
	\end{cor}
	\begin{proof}
		If $\lambda$ is an eigenvalue of $\tilde{\mathbf{L}}_{d}$, then there exists a smooth function $\varphi\in  C^{\infty}(\overline{\mathbb{B}})$ such that $e^{\lambda s} \varphi(y)$ solves the linearized equation \eqref{eq:linearized-ss}. By the Lorentz transformation and Proposition \ref{prop-transform-linear-op}, we have
		$$\mathcal{T}_{d}(e^{\lambda s} \varphi)(s',y') = e^{\lambda s'} \left( 1-d \cdot y'\right)^{-\lambda} \left(\frac{1}{\gamma(1-\beta\cdot y')}\right)^{\frac{2}{p-1}} \varphi\left( \frac{y'-\gamma d + \frac{\gamma^2 d\cdot y'}{1+\gamma}d}{\gamma(1-d \cdot y')} \right) := e^{\lambda s'} \psi(y') $$
		solves the linearized equation \eqref{linearized-ss-transform}, where $\gamma = \frac{1}{\sqrt{1-|d|^2}}$. Thus, $\psi(y')\in  C^{\infty}(\overline{\mathbb{B}})$ satisfies \eqref{eigen-eq} with $d=0$, i.e., $\lambda$ is also an eigenvalue of $\tilde{\mathbf{L}}_{0}$. Conversely, by the inverse Lorentz transformation, any eigenvalue of $\tilde{\mathbf{L}}_{0}$ is also an eigenvalue of $\tilde{\mathbf{L}}_{d}$. 
	\end{proof}
	
	\begin{remark}
		According to the proof of Corollary \ref{cor-eigen}, the following transformation (induced by the Lorentz transformation $\mathcal{T}_{d}$)
		\begin{align}\label{transform-eigen-eq}
			\psi(y') =\left( 1-d \cdot y'\right)^{-\lambda} \left(\frac{1}{\gamma(1-\beta\cdot y')}\right)^{\frac{2}{p-1}} \varphi\left( \frac{y'-\gamma d + \frac{\gamma^2 d\cdot y'}{1+\gamma}d}{\gamma(1-d \cdot y')} \right), \quad \gamma = \frac{1}{\sqrt{1-|d|^2}}
		\end{align}
		actually gives the transformation of the eigen-equation \Cref{eigen-eq}  to itself with $d=0$.
	\end{remark}
	Thus, the mode stability of $\kappa_d$ is equivalent to that of $\kappa_0$, the latter of which has already been established in a series of works by Donninger and his collaborators \cite{donninger2012stable,donninger2016blowup}; see also Ostermann \cite[Lemma 3.3]{ostermann2024stable}.
	\begin{lem}\label{lem-mode-stability-0} $\kappa_0$ is mode stable, i.e. there exist no eigenvalues of $\tilde{\mathbf{L}}_{0}$ such that $\operatorname{Re} \lambda >-1$, expect $0$ and $1$.
	\end{lem}
	As a corollary, we obtain
	\begin{cor}\label{prop-mode-stability} For all $|d|<1$, $\kappa_d$ is mode stable, i.e. there exist no eigenvalues of $\tilde{\mathbf{L}}_{d}$ such that $\operatorname{Re} \lambda >-1$, expect $0$ and $1$.
	\end{cor}


	
	\section{Spectral equivalence and linear stability}\label{sec:linear}
   In this section, we establish the equivalence of the full spectra of the linearized operators around each $\kappa_d$ and, consequently, derive the linear stability of the self-similar profile $\kappa_d$. This analysis builds upon the spectral strategy developed in \cite{ostermann2024stable} for the ODE blow-up profile $\kappa_0$. Recall the first-order linearized system:
   	\begin{align*}
   	\partial_s \begin{pmatrix}
   		q_1 \\ q_2
   	\end{pmatrix} & = \begin{pmatrix}
   		-y\cdot \nabla q_1 - \frac{2}{p-1}q_1 + q_2\\
   		\Delta q_1 + p|\kappa_d|^{p-1} q_1  -y\cdot \nabla q_2 - \frac{p+1}{p-1}q_2 
   	\end{pmatrix} =:  \tilde{\mathbf{L}}_d  \begin{pmatrix}
   		q_1 \\ q_2
   	\end{pmatrix}.
   \end{align*}

   \subsection{Functional setup}
   We introduce the following standard Hilbert space.
   \begin{defn}\label{SobolevSpaces}
   	Let $(N,k,R)\in\NN\times\NN\times\RR_{>0} $. We define a Hilbert space by
   	\begin{equation*}
   		\mathcal{H}^{k}(\BB^{N}_R) \coloneqq H^{k}(\BB^{N}_R) \times H^{k-1}(\BB^{N}_R) \,, \qquad
   		\| \mathbf{f} \|_{\mathcal{H}^{k}(\BB^{N}_R)} \coloneqq \| (f_{1},f_{2}) \|_{H^{k}(\BB^{N}_R) \times H^{k-1}(\BB^{N}_R)} \,.
   	\end{equation*}
   	We also define for $\delta>0$ a closed ball
   	\begin{equation*}
   		\mathcal{H}^{k}_{\delta}(\BB^{N}_R) \coloneqq \{ \mathbf{f} \in \mathcal{H}^{k}(\BB^{N}_R) \mid \| \mathbf{f} \|_{\mathcal{H}^{k}(\BB^{N}_R)} \leq \delta \} \,.
   	\end{equation*}
   \end{defn}
	Then, we define the free wave operator in terms of self-similar variables.
	\begin{defn}[Free wave operator]\label{WaveFlowOperation}
		Let $(N,p,k,R)\in\NN\times\RR_{>1}\times \NN\times\RR_{\ge 1}$. For $\mathbf{f}\in C^{\infty}(\overline{\BB^{N}_R})^{2}$ we define
		\begin{align*}
			\mathbf{L}_{N,p,R}\mathbf{f} \in C^{\infty}(\overline{\BB^{N}_R})^{2}
			\quad\text{by}\quad
			(\mathbf{L}_{N,p,R} \mathbf{f})(y) =
			 \begin{pmatrix}
				-y\cdot \nabla f_1 - \frac{2}{p-1}f_1 + f_2\\
				\Delta f_1  -y\cdot \nabla f_2 - \frac{p+1}{p-1}f_2 
			\end{pmatrix}.
		\end{align*}
		 The operator $\mathbf{L}_{N,p,k,R}: \dom(\mathbf{L}_{N,p,k,R}) \subset \mathcal{H}^{k}(\BB^{N}_R) \rightarrow \mathcal{H}^{k}(\BB^{N}_R)$ is densely defined by
		\begin{equation*}
			\mathbf{L}_{N,p,k,R}\mathbf{f} = \mathbf{L}_{N,p,R}\mathbf{f} \,,
			\qquad
			\dom(\mathbf{L}_{N,p,k,R}) = C^{\infty}(\overline{\BB^{N}_R})^{2} \,.
		\end{equation*}
	\end{defn}
	The closure of the free wave operator generates a strongly continuous semigroup, see \cite[Theorem 2.1]{ostermann2024stable} for the detailed proof.
	\begin{thm}\label{FreeSemigroupTHM}
		Let $(N,p,k,R)\in\NN\times\RR_{>1}\times \NN\times\RR_{\ge 1}$. The operator $\mathbf{L}_{N,p,k,R}$ is closable and its closure $\overline{\mathbf{L}_{N,p,k,R}}$ is the generator of a strongly continuous operator semigroup
		\begin{equation*}
			\mathbf{S}_{N,p,k,R}: \RR_{\geq 0} \rightarrow \mathcal{L} \left(  	\mathcal{H}^{k}(\BB^{N}_R) \right)
		\end{equation*}
		which satisfies that for any $0<\varepsilon<\frac{1}{2}$ there is a constant $M_{N,p,k,R,\varepsilon}\geq 1$ such that
		\begin{equation*}
			\| \mathbf{S}_{N,p,k,R}(s) \mathbf{f} \|_{	\mathcal{H}^{k}(\BB^{N}_R)} \leq M_{N,p,k,R,\varepsilon} e^{\omega_{N,p,k,R,\varepsilon} s} \| \mathbf{f} \|_{	\mathcal{H}^{k}(\BB^{N}_R)}
		\end{equation*}
		for all $\mathbf{f}\in 	\mathcal{H}^{k}(\BB^{N}_R)$ and all $s\geq 0$, where
		\begin{equation*}
			\omega_{N,p,k,\varepsilon} = -s_{p} + \max\Big\{ \frac{N}{2}-k,\varepsilon \Big\}
			\qquad\text{and}\qquad
			s_{p} = \frac{2}{p-1} \,.
		\end{equation*}
	\end{thm}
	
	Then, we introduce the potential operator coming from the linearization around $\kappa_d$. 
	\begin{defn}
		\label{Potential}
		Let $d\in \BB^N$ and $(N,p,k,R)\in \NN\times\RR_{>1}\times\NN\times\RR_{\ge 1}$ with $R<|d|^{-1}$. Let $V_{d}\in C^{\infty}(\overline{\BB^{N}_R})$ be given by
		\begin{equation*}
			V_{d}(y) = p|\kappa_d|^{p-1}= (s_{p}+1) (s_{p}+2) \left( 1-|d|^2 \right) \left( 1 + d\cdot y \right)^{-2} \,.
		\end{equation*}
		We define the bounded linear operator
		\begin{equation*}
			\mathbf{L}_{d}' \in \mathcal{L} \left( \mathcal{H}^{k}(\BB^{N}_R) \right)
			\qquad\text{by}\qquad
			\mathbf{L}_{d}'\mathbf{f} =
			\begin{bmatrix}
				0 \\
				V_{d} f_{1}
			\end{bmatrix}
			\,.
		\end{equation*}
	\end{defn}
	Now, the full linearized operator is composed of a perturbation of the free wave operator.
	\begin{defn}
		\label{LinearWaveEvolutionOperator}
		Let $d\in \BB^N$ and $(N,p,k,R)\in \NN\times\RR_{>1}\times\NN\times\RR_{\ge 1}$ with $R<|d|^{-1}$. The closed linear operator $\mathbf{L}_{d}: \dom(\mathbf{L}_{d})\subset\mathcal{H}^{k}(\BB^{N}_R)\rightarrow\mathcal{H}^{k}(\BB^{N}_R)$ is densely defined by
		\begin{equation*}
			\mathbf{L}_{d} = \overline{ \mathbf{L}_{N,p,k} } + \mathbf{L}_{d}' \,, \qquad \dom(\mathbf{L}_{d}) = \dom(\overline{\mathbf{L}_{N,p,k}}) \,.
		\end{equation*}
	\end{defn}
	Perturbation theory for semigroups yields that this operator is indeed the generator of the linearized wave flow.
	\begin{prop}[{\cite[Proposition 3.1]{ostermann2024stable}}]
		\label{SGbeta}
		Let $d\in \BB^N$ and $(N,p,k,R)\in \NN\times\RR_{>1}\times\NN\times\RR_{\ge 1}$ with $R<|d|^{-1}$. The closed linear operator $\mathbf{L}_{d} :\dom(\mathbf{L}_{d}) \subset \mathcal{H}^{k}(\BB^{N}_R) \rightarrow \mathcal{H}^{k}(\BB^{N}_R)$ is the generator of a strongly continuous operator semigroup
		\begin{equation*}
			\mathbf{S}_{d}: \RR_{\geq 0} \rightarrow \mathcal{L}\left(\mathcal{H}^{k}(\BB^{N}_R)\right) \,.
		\end{equation*}
		Moreover, let $0<\varepsilon<\frac{1}{2}$ and $M_{N,p,k,\varepsilon}\geq 1$, $\omega_{N,p,k,\varepsilon}\in\RR$ be the constants from \Cref{FreeSemigroupTHM}. Then the bound
		\begin{equation*}
			\| \mathbf{S}_{d}(s) \|_{\mathcal{L}\left(\mathcal{H}^{k}(\BB^{N}_R)\right)} \leq M_{N,p,k,\varepsilon} \ee^{ \Big( \omega_{d,p,k,\varepsilon} + M_{N,p,k,\varepsilon} \| \mathbf{L}_{d}' \|_{\mathcal{L}\left(\mathcal{H}^{k}(\BB^{N}_R)\right)} \Big) s}
		\end{equation*}
		holds for all $s\geq 0$ and all $|d|<1$.
	\end{prop}
	
	\subsection{Spectral analysis of the linearized operator}

	The spectrum of the linearized operator $\mathbf{L}_0$ has already been known in \cite[Lemma 3.3]{ostermann2024stable}
	\begin{lem}
		\label{SpecL0}
		Let $(N,p,k,R)\in \NN \times\RR_{>1}\times\NN\times\RR_{\ge 1}$ such that $\frac{N}{2}-s_{p}-k < 0$. Put
		\begin{equation*}
			\omega_{N,p,k} \coloneqq \max\left\{-1,\frac{N}{2}-s_{p}-k,-s_{p} \right\}
		\end{equation*}
		and let $\omega_{N,p,k}<\omega_{0}<0$ be arbitrary but fixed. We have
		\begin{equation*}
			\sigma(\mathbf{L}_{0}) \cap \overline{\mathbb{H}_{\omega_{0}}} = \{0,1\} \,.
		\end{equation*}
		Both points $0,1\in\sigma(\mathbf{L}_{0})$ belong to the point spectrum of $\mathbf{L}_{0}$ with eigenspaces
		\begin{equation*}
			\ker(\mathbf{L}_{0}) = \langle \mathbf{f}_{0,0,i} \rangle_{i=1}^{N}
			\qquad\text{and}\qquad
			\ker(\mathbf{I}-\mathbf{L}_{0}) = \langle \mathbf{f}_{1,0} \rangle
		\end{equation*}
		spanned by the symmetry modes $\mathbf{f}_{0,0,i},\mathbf{f}_{1,0}\in C^{\infty}(\overline{\BB^{N}_R})^{2}$ given by
		\begin{equation*}
			\mathbf{f}_{0,0,i}(y) =
			\begin{bmatrix}
				s_{p}\kappa_0y_{i} \\
				s_{p}(s_{p}+1)\kappa_0 y_{i}
			\end{bmatrix}
			\qquad\text{and}\qquad
			\mathbf{f}_{1,0}(y) =
			\begin{bmatrix}
				s_{p}\kappa_0 \\
				s_{p}(s_{p}+1)\kappa_0
			\end{bmatrix}
			\,,\qquad
			i=1,\ldots,N.
		\end{equation*}
	\end{lem}

	Now, we establish the following spectral equivalence of linearized operators. 
	\begin{prop}[Spectral equivalence]
		\label{SpectralEquivalence}
		Let $d\in \BB^N$ and $(N,p,k,R)\in \NN\times\RR_{>1}\times\NN\times\RR_{\ge 1}$ such that $\frac{N}{2}-s_{p}-k < 0$ and $R<|d|^{-1}$, we have $ \sigma(\mathbf{L}_{d}) =   \sigma(\mathbf{L}_{0}) \,.$	
	\end{prop}
	\begin{proof}
		By definition, $\mathbf{L}_{d} =\mathbf{L}_{0} -  \mathbf{L}_{0}' + \mathbf{L}_{d}'$, where the potential operators $\mathbf{L}_{0}'$ and $\mathbf{L}_{d}'$ are compact, see \cite[Lemma 3.1]{ostermann2024stable}. Thus, by \cite[p. 244, Theorem 5.35 ]{kato1995} the essential spectrum of $\mathbf{L}_{0}$ and $\mathbf{L}_{d}$ are the same. Besides,  by \Cref{cor-eigen} and \Cref{transform-eigen-eq} (note that the transformation in \Cref{transform-eigen-eq} maps $H^k(\BB^N)$ to $H^k(\BB^N)$ itself), we have $\sigma_p( \mathbf{L}_{d}) = \sigma_p( \mathbf{L}_{0})$. Therefore, $ \sigma(\mathbf{L}_{d}) =   \sigma(\mathbf{L}_{0}) \,.$	
	\end{proof}
	
	As a corollary, we obtain the spectrum of $\mathbf{L}_{d}$, which extends \cite[Theorem 3.1]{ostermann2024stable} for $\mathbf{L}_{d}$ with $|d|\ll 1$ to the whole range $|d|<1$. 
	\begin{thm}
		\label{SpecThmLbeta}
		Let $d\in \BB^N$ and $(N,p,k,R)\in \NN\times\RR_{>1}\times\NN\times\RR_{\ge 1}$ such that $\frac{N}{2}-s_{p}-k < 0$ and $R<|d|^{-1}$. Put
		\begin{equation*}
			\omega_{N,p,k} \coloneqq \max\left\{-1,\frac{N}{2}-s_{p}-k,-s_{p} \right\}
		\end{equation*}
		and let $\omega_{N,p,k}<\omega_{0}<0$ be arbitrary but fixed. Then,
		\begin{equation*}
			\sigma(\mathbf{L}_{d}) \cap \overline{\mathbb{H}_{\omega_{0}}} = \{0,1\} \,.
		\end{equation*}
		Moreover, both points $0,1\in\sigma(\mathbf{L}_{d})$ belong to the point spectrum of $\mathbf{L}_{d}$ with eigenspaces
		\begin{equation*}
			\ker(\mathbf{L}_{d}) = \langle \mathbf{f}_{0,d,i} \rangle_{i=1}^{d}
			\qquad\text{and}\qquad
			\ker(\mathbf{I}-\mathbf{L}_{d}) = \langle \mathbf{f}_{1,d} \rangle
		\end{equation*}
		spanned by the symmetry modes $\mathbf{f}_{0,d,i},\mathbf{f}_{1,d}\in C^{\infty}(\overline{\BB^{N}_R})^{2}$ from \Cref{EV}. Furthermore, the algebraic multiplicity of both eigenvalues $\lambda\in\{0,1\}$ is finite and equal to the respective geometric multiplicity. In particular,
		\begin{equation*}
			\ran(\mathbf{P}_{\lambda,d}) = \ker(\lambda\mathbf{I}-\mathbf{L}_{d})
		\end{equation*}
		where $\mathbf{P}_{\lambda,d} \in \mathcal{L}\left( \mathcal{H}^{k}(\BB^{N}_R)\right)$, defined by
		\begin{equation*}
			\mathbf{P}_{\lambda,d} \coloneqq \frac{1}{2\pi\ii}\int_{\pd\mathbb{D}_{r}(\lambda)} \mathbf{R}_{\mathbf{L}_{d}}(z) \dd z \,,
		\end{equation*}
		is the \emph{Riesz projection} associated to the isolated eigenvalue $\lambda\in\{0,1\}$ of $\mathbf{L}_{d}$, respectively.
	\end{thm}
	\begin{proof}
		By \Cref{SpectralEquivalence} and \Cref{SpecL0}, we obtain  $\sigma(\mathbf{L}_{d}) \cap \overline{\mathbb{H}_{\omega_{0}}} = \{0,1\} \,.$ The properties of eigenspaces follow from \Cref{EV},  \Cref{SpecL0} and the transformation \Cref{transform-eigen-eq}. The properties of the projection $\mathbf{P}_{\lambda,d}$ follow from the proof of \cite[Theorem 3.1]{ostermann2024stable}.
	\end{proof}
	
	\subsection{Uniform growth bounds} 
	
	\begin{prop} \label{prop-uniform-bound}
		Let $d_0 \in \BB^N$ and $(N,p,k,R)\in \NN\times\RR_{>1}\times\NN\times\RR_{\ge 1}$ such that $\frac{N}{2}-s_{p}-k < 0$ and $R<|d_0|^{-1}$, 
		and let $\omega_{N,p,k}<\omega_{0}<0$ be arbitrary but fixed. Then, there exist constants $R_0>R/(1-|d_0|R)$ and $C>0$ such that 
		\begin{equation*}
			\| \mathbf{R}_{\mathbf{L}_{d}}(z) \|_{\mathcal{L}\left(\mathcal{H}^{k}(\BB^{N}_R)\right)} \leq C
		\end{equation*}
		holds for all $z\in\overline{\mathbb{H}_{\omega_{0}}}\setminus \Big( \mathbb{D}_{\frac{|\omega_{0}|}{2}}(0) \cup \mathbb{D}_{\frac{|\omega_{0}|}{2}}(1) \Big)$ and all $d\in\overline{\BB^{N}_{R_0^{-1}}(d_0)}$.
	\end{prop}
\begin{proof}
	The proof follows similarly to \cite[Proposition 3.2]{ostermann2024stable}, with $\mathbf{L}_0$ replaced by $\mathbf{L}_{d_0}$. Hence, we omit the details.
\end{proof}
	
	Then, we can also derive the uniform growth bounds of the semigroup $\mathbf{S}_d$ for all $d \in \overline{B^N_{R_0^{-1}}(d_0)}$.
	\begin{thm}
		\label{stableEvo}
		Let $d_0 \in \BB^N$ and $(N,p,k,R)\in \NN\times\RR_{>1}\times\NN\times\RR_{\ge 1}$ such that $\frac{N}{2}-s_{p}-k < 0$ and $R<|d_0|^{-1}$.  Let $ d \in \overline{B^N_{R_0^{-1}}}(d_0)$ with $R_0>R/(1-|d_0|R)$. The semigroup $\mathbf{S}_{d}(s)$, introduced in \Cref{SGbeta}, and the projections $\mathbf{P}_{0,d},\mathbf{P}_{1,d} \in \mathcal{L}\left(\mathcal{H}^{k}(\BB^{N}_R)\right)$, introduced in \Cref{SpecThmLbeta}, satisfy
		\begin{align*}
			&&
			\mathbf{P}_{0,d}\mathbf{S}_{d}(s) &= \mathbf{S}_{d}(s)\mathbf{P}_{0,d} = \mathbf{P}_{0,d} \,,
			&
			\mathbf{P}_{0,d}\mathbf{P}_{1,d} &= \mathbf{0} \,, & \\
			&&
			\mathbf{P}_{1,d}\mathbf{S}_{d}(s) &= \mathbf{S}_{d}(s)\mathbf{P}_{1,d} = \ee^{s}\mathbf{P}_{1,d} \,,
			&
			\mathbf{P}_{1,d}\mathbf{P}_{0,d} &= \mathbf{0} \,. &
		\end{align*}
		Moreover, for each arbitrary but fixed $\omega_{N,p,k}<\omega_{0}<0$ there exist constants $R_0>R/(1-|d_0|R)$ and $M_{\omega_{0}}\geq 1$ such that
		\begin{equation*}
			\| \mathbf{S}_{d}(s) ( \mathbf{I} - \mathbf{P}_{d} ) \|_{\mathcal{L}\left(\mathcal{H}^{k}(\BB^{N}_R)\right)} \leq M_{\omega_{0}} \ee^{\omega_{0}s} \| \mathbf{I} - \mathbf{P}_{d} \|_{\mathcal{L}\left(\mathcal{H}^{k}(\BB^{N}_R)\right)}
		\end{equation*}
		for all $s\geq 0$ and all $d \in \overline{B^N_{R_0^{-1}}}(d_0)$.
	\end{thm}
	\begin{proof}
		This result follows directly from \Cref{SpecThmLbeta}, \Cref{prop-uniform-bound}, and the proof of \cite[Theorem 3.2]{ostermann2024stable}.
	\end{proof}

\section{Proof of \Cref{MainTHM} and \Cref{thm-trapping}}\label{sec:nonlinear}
Now, we consider the full nonlinear Cauchy problem \eqref{CauchyProblem}. In terms of the self-similar variables
	\begin{align*}
	u(t, x) = \frac{1}{\left(T-t\right)^{\frac{2}{p-1}}}  U_{T,0}(s, y), \quad s=-\log \left(T-t\right) + \log T, \quad y=\frac{x}{T-t},
\end{align*}
and $U(s,y) = \kappa_d(y) + \eta(s,y)$, \eqref{CauchyProblem} is equivalent to
\begin{align}\label{nlw-eta}
	\begin{aligned}
		&\partial_{ss} \eta + \frac{p+3}{p-1} \partial_s \eta + 2 y\cdot\nabla \partial_s \eta +   \frac{2(p+1)}{p-1} y\cdot\nabla \eta + (y_i y_j - \delta_{ij}) \partial_i \partial_j \eta +\left(\frac{2(p+1)}{(p-1)^2} -  p|\kappa_d|^{p-1} \right) \eta \\
		& = |\kappa_d+\eta|^{p-1}(\kappa_d+\eta) - \kappa_d^p - p |\kappa_d|^{p-1}\eta,  \quad s>0, |y|\le R,
	\end{aligned}
\end{align}
with the initial data
\begin{equation}\label{InitialData-eta}
	\begin{aligned}
		&\eta(0,y) = T^{\frac{2}{p-1}}(u_{1,0,d_0}^{*}(0,Ty) + f(Ty) ) - \kappa_d(y) \,, \quad |y|\le R  \,,\\
		& \left(\pd_s +y\cdot \nabla y + \frac{2}{p-1}\right)\eta(0,y) = T^{\frac{p+1}{p-1}}\left(\pd_{t}u_{1,0,d_0}^{*}(0,Ty) + g(Ty) \right) -  \left(y\cdot \nabla + \frac{2}{p-1}\right)\kappa_d(y) \,,  \quad |y|\le R  \,,
	\end{aligned}
\end{equation}
Rewriting \eqref{nlw-eta}-\eqref{InitialData-eta} as a first-order system, we have 
\begin{equation}
	\label{abstractCauchy}
	\renewcommand{\arraystretch}{1.2}
	\left\{
	\begin{array}{rcl}
		\pd_{s}\mathbf{q}(s,\,.\,) &=& \mathbf{L}_{d}\mathbf{q}(s,\,.\,) + \mathbf{N}_{d}(\mathbf{q}(s,\,.\,)) \,, \\
		\mathbf{q}(0,\,.\,) &=&  \mathbf{f}^{T} + \mathbf{f}_{0}^{T} - \mathbf{f}_{d} \,,
	\end{array}
	\right.
\end{equation}
where
\begin{equation}
	\label{Nonlin}
	\mathbf{N}_{d}(\mathbf{q}(s,\,.\,))(y) =
	\begin{bmatrix}
		0\\
		|\kappa_d(y)+q_1(s,y)|^{p-1}(\kappa_d(y)+q_1(s,y)) - \kappa_d(y)^p - p |\kappa_d(y)|^{p-1}q_1(s,y),
	\end{bmatrix}
\end{equation}
and initial data 
\begin{align*}
	&&
	\mathbf{f}^{T}(y) &=
	\begin{bmatrix}
		T^{s_{p}} f(Ty) \\
		T^{s_{p}+1} g(Ty)
	\end{bmatrix}
	\,, &
	\mathbf{f}_{0}^{T}(y) &=
	\begin{bmatrix}
		T^{s_{p}} \kappa_0 \gamma(d_0)^{-s_p} (1+ T d_0\cdot y)^{-s_p} \\
		T^{s_{p}+1} s_p \kappa_0 \gamma(d_0)^{-s_p} (1+ T d_0\cdot y)^{-s_p-1} 
	\end{bmatrix}
	\,, & \\
	&&
	\mathbf{f}_{d}(y) &=
	\begin{bmatrix}
		\kappa_0 \gamma(d)^{-s_p} (1+  d\cdot y)^{-s_p}\\
		s_p \kappa_0 \gamma(d)^{-s_p} (1+ d\cdot y)^{-s_p-1} 
	\end{bmatrix}
	\,,
	&& &
\end{align*}

Since we have derived the linear stability of $\kappa_d$ along its stable subspace by \Cref{stableEvo}, the nonlinear stability follows from the same strategy in \cite{ostermann2024stable}.  More precisely, we have
	\begin{prop}
	\label{StableNonlinFlowClassic}
	Let $d_0\in \BB^N$ and $(N,p,k,R)\in\NN\times\RR_{>1}\times \NN\times\RR_{\ge 1}$ such that $k> \frac{N}{2}$ and $R<|d_0|^{-1}$. Let $0 < \varepsilon < \omega_{p}$. There are constants $0<\delta_{\varepsilon}<1$, $C_{\varepsilon} > 1$ such that for all $0<\delta\leq\delta_{\varepsilon}$, $C\geq C_{\varepsilon}$ and all real-valued $\mathbf{f}\in C^{\infty}(\RR^{N})^{2}$ with $\| \mathbf{f} \|_{\mathcal{H}^{k}(\RR^{N})} \leq \frac{\delta}{C^{2}}$, there are parameters $d^{*}\in\overline{\BB^{N}_{\frac{\delta}{C}}(d_0)}$, $T^{*}\in\overline{\BB^{1}_{\frac{\delta}{C}}(1)}$ and a unique $\mathbf{q}_{d^{*},T^{*}} \in C^{\infty}\big(\overline{(0,\infty)\times\BB^{N}_R}\big)^{2}$ such that $\|\mathbf{q}_{d^{*},T^{*}}(s,\,.\,) \|_{\mathcal{H}^{k}(\BB^{N}_R)} \leq \delta \ee^{(-\omega_{p}+\varepsilon)s}$ and
	\begin{equation}
		\label{CauchyClassical}
		\renewcommand{\arraystretch}{1.2}
		\left\{
		\begin{array}{rcl}
			\pd_{s}\mathbf{q}_{d^{*},T^{*}}(s,\,.\,) &=& \mathbf{L}_{d^{*}}\mathbf{q}_{d^{*},T^{*}}(s,\,.\,) + \mathbf{N}_{d^{*}}(\mathbf{q}_{d^{*},T^{*}}(s,\,.\,)) \,, \\
			\mathbf{q}_{d^{*},T^{*}}(0,\,.\,) &=&  \mathbf{f}^{T^{*}} + \mathbf{f}_{0}^{T^{*}} - \mathbf{f}_{d^{*}} \,,
		\end{array}
		\right.
	\end{equation}
	for all $s\geq 0$.
\end{prop}
\begin{proof}
	The proof follows the argument in \cite[Proposition 4.3]{ostermann2024stable} with minor modifications. For completeness, the detailed proof is provided in the appendix.
\end{proof}

Now we are ready to prove  \Cref{MainTHM} and \Cref{thm-trapping}.
\begin{proof}[Proof of \Cref{MainTHM}]
	Fix $0<\varepsilon<\omega_{p}$. From \Cref{StableNonlinFlowClassic}, we pick $0<\delta_{\varepsilon}<1$ and $C_{\varepsilon}\geq 1$. Then, for any $0<\delta\leq\delta_{\varepsilon}$, $C\geq C_{\varepsilon}$ and real-valued $(f,g)\in C^{\infty}(\RR^{N})^{2}$ with $\| (f,g) \|_{H^{k}(\RR^{N})\times H^{k-1}(\RR^{N})} \leq \frac{\delta}{C^{2}}$ there exists a unique $( q_{T^{*},0, d^{*},1},q_{T^{*},0, d^{*},2} ) \in C^{\infty}\big(\overline{(0,\infty)\times\BB^{N}_R}\big)^{2}$ that is a jointly smooth classical solution to the abstract Cauchy problem \eqref{CauchyClassical}. Denote
	\begin{equation*}
		u \coloneqq u_{T^{*}, 0, d^{*}}^{*} + v \in C^{\infty}\big(\Omega^{1,N}(T^{*})\big),
	\end{equation*}
	where $v \in C^{\infty}\big(\Omega^{1,N}(T^{*})\big)$ given by
	\begin{equation*}
		v(t,x) \coloneqq (T^{*}-t)^{-s_{p}} q_{T^{*},0,d^{*},1} \Big( \log \tfrac{T^{*}}{T^{*}-t}, \tfrac{x}{T^{*}-t} \Big)
	\end{equation*}
	is related to $q_{T^{*},0,d^{*},2}$ via
	\begin{equation*}
		\pd_{t} v(t,x) = (T^{*}-t)^{-s_{p}-1} q_{T^{*},0, d^{*},2} \Big( \log \tfrac{T^{*}}{T^{*}-t}, \tfrac{x}{T^{*}-t} \Big).
	\end{equation*}
	Then, $u(t,x)$ is the unique solution to the Cauchy problem \eqref{CauchyProblem} posed in \Cref{MainTHM}, 
	Moreover, by scaling of the homogeneous seminorms, we infer from \Cref{StableNonlinFlowClassic} the bounds
	\begin{align*}
		(T^{*}-t)^{ - \frac{N}{2} + s_{p} + s }\| u(t,\,.\,) - u_{T^{*}, 0, d^{*}}^{*}(t,\,.\,) \|_{\dot{H}^{s}(\BB^{N}_{R(T^{*}-t)})} &=
		\| q_{1} \Big( \log \tfrac{T^{*}}{T^{*}-t}, \,.\, \Big) \|_{\dot{H}^{s}(\BB^{N}_R)} \\&\leq
		\delta \left( \tfrac{T^{*}-t}{T^{*}} \right)^{\omega_{p}-\varepsilon}
	\end{align*}
	for $s=0,1,\ldots,k$, and
	\begin{align*}
		&
		(T^{*}-t)^{ - \frac{N}{2} + s_{p} + s } \| \pd_{t}u(t,\,.\,) - \pd_{t}u_{T^{*},0,d^{*}}^{*}(t,\,.\,) \|_{\dot{H}^{s-1}(\BB^{N}_{R(T^{*}-t)})} \\&\indent=
		\| q_{2} \Big( \log \tfrac{T^{*}}{T^{*}-t}, \,.\, \Big) \|_{\dot{H}^{s-1}(\BB^{N}_R)} \\&\indent\leq
		\delta \left( \tfrac{T^{*}-t}{T^{*}} \right)^{\omega_{p}-\varepsilon}
	\end{align*}
	for $s=1,\ldots,k$.
\end{proof}

\begin{proof}[Proof of \Cref{thm-trapping}] Note that for any solution $U$ of equation \eqref{nlw-ss-s2}, by the time translation invariance, 
\begin{align*}
	\tilde{U}(s', y'):= \left(1+b e^{s'}\right)^{-\frac{2}{p-1}}U\left(s'-\log \left(1+b e^{s'}\right)+\log(1+b), \frac{y'}{1+b e^{s'}}\right);
\end{align*}
also solves \eqref{nlw-ss-s2} for any $b\in \mathbb{R}$. In particular, taking $b = T^{-1}-1$ and denoting $\tilde{U}_T(s', y')= \tilde{U}(s', y')$, we have 
\begin{align*}
	\tilde{U}_T(0, y')= T^{s_p}U\left(0, Ty'\right).
\end{align*}
Denote 
$$f(y) := U(0,y) - \kappa_{d_0}(y)$$
and
$$g(y) := (\pd_s+ y\cdot\nabla +s_p)(U(s,y) -\kappa_{d_0}(y)) \mid_{s=0}.$$ 
Let $\tilde{U}(s',y') = \kappa_{d}(y') + \tilde{\eta}(s',y')$ and $\tilde{\textbf{q}}= (\tilde{q}_1, \tilde{q}_2)^\perp  = (\tilde{\eta},  (\pd_{s'}+ y'\cdot\nabla +s_p) \tilde{\eta} )^\perp$, then $\tilde{\textbf{q}} $ solves the Cauchy problem \eqref{abstractCauchy}. By \Cref{StableNonlinFlowClassic}, there exists constants $\epsilon>0, 0<\delta<1, C>1$ such that for all $\|(f,g)\|_{\mathcal{H}^k(\BB^{N})}\le \delta$, there are parameters $d^{*}\in\overline{\BB^{N}_{\frac{\delta}{C}}(d_0)}$, $T^{*}\in\overline{\BB^{1}_{\frac{\delta}{C}}(1)}$ and a unique solution $\tilde{\mathbf{q}}_{d^{*},T^{*}} \in C^{\infty}\big(\overline{(0,\infty)\times\BB^{N}}\big)^{2}$ 
to \Cref{CauchyClassical} satisfying $\|\tilde{\mathbf{q}}_{d^{*},T^{*}}(s',\,.\,) \|_{\mathcal{H}^{k}(\BB^{N})} \leq \delta \ee^{(-\omega_{p}+\varepsilon)s'}$
for all $s'\geq 0$. By the uniqueness, we have 
$$ \tilde{U}_{T^*}(s',y') = \kappa_{d^*}(y') +  \tilde{q}_{d^{*},T^{*},1}(s',y').$$
Let $b^* = (T^*)^{-1} -1$, $s = s' - \log(1+b^*e^{s'}) +\log (1+b^*), y = \frac{y'}{1+b^*e^{s'}}$, then 
\begin{align*}
	U(s,y) &=  (1+b^* e^{s'})^{\frac{2}{p-1}}\tilde{U}_{T^*}(s',y')\\
	& =   \left( \frac{1+b^*}{1+b^* -b^* e^{s}}\right)^{\frac{2}{p-1}}\tilde{U}_{T^*}\left(s - \log(1+b^* - b^* e^{s}), y \left( \frac{1+b^*}{1+b^* -b^* e^{s}}\right)\right).
\end{align*}
Therefore, we have
\begin{enumerate}[1)]
	\item if $b^*>0$, then $U(s,y)$ is not well defined for all $(s,y) \in [0,+\infty)\times \BB^N$, 
	\item if $b^*<0$, then		
	\begin{align*}
		\left\| \binom{U(s)}{\partial_s U(s)} \right\|_{{H}^k(\BB^N)\times {H}^{k-1}(\BB^N)} \rightarrow 0 \text { as } s \rightarrow \infty \text { exponentially fast }
	\end{align*}
	\item if $b^*=0$, then $U(s,y) = \tilde{U}_{T^*}(s,y)$, which implies that		
	\begin{align*}
		\forall s \geq 0, \  \left\|\binom{U(s)}{\partial_s U(s)}- \binom{\kappa_{d^*}(y)}{0} \right\|_{{H}^k(\BB^N)\times {H}^{k-1}(\BB^N)} \leq C \delta  e^{-(\omega_{p}+\varepsilon)s}.
	\end{align*}
\end{enumerate}
\end{proof}

\section*{Acknowledgment}
The author would like to thank Professor Tej-eddine Ghoul, Professor Nader Masmoudi, and Professor Hatem Zaag for valuable discussions and insightful comments that have greatly influenced this work.

\appendix	
	\section{Nonlinear stability}

	With the semigroup from \Cref{SGbeta} at hand, Duhamel's principle yields a strong formulation
	\begin{equation}
		\label{NonLinDuhamel}
		\mathbf{q}(s,\,.\,) = \mathbf{S}_{d}(s)\mathbf{q}(0,\,.\,) + \int_{0}^{s} \mathbf{S}_{d}(s-s')\mathbf{N}_{d}(\mathbf{q}(s',\,.\,)) \dd s'
	\end{equation}
	of problem \eqref{abstractCauchy}. 
	
	\subsection{Properties of the nonlinearity}
	The following nonlinear estimates can be derived by a trivial extension of \cite[Lemma 4.2]{ostermann2024stable}. 
	\begin{lem}
		\label{NonLinLip}
		Let $d_0 \in \BB^N$ and $(N,p,k,R)\in \NN\times\RR_{>1}\times\NN\times\RR_{\ge 1}$ such that $k>\frac{N}{2}$ and $R<|d_0|^{-1}$.  Let $R_0>R/(1-|d_0|R)$. There exists a constant $\delta>0$ such that the map
		\begin{equation*}
			\mathbf{N}_{d}: \mathcal{H}^{k}_{\delta}(\BB^{N}_R) \rightarrow \mathcal{H}^{k}(\BB^{N}_R) \,, \qquad \mathbf{f} \mapsto \mathbf{N}_{d}(\mathbf{f}) \,,
		\end{equation*}
		is well-defined with
		\begin{equation*}
			\mathbf{N}_{d}(\mathbf{f}) \in \mathcal{H}^{k+1}(\BB^{N}_R) \subset \mathcal{H}^{k}(\BB^{N}_R)
		\end{equation*}
		for all $\mathbf{f}\in \mathcal{H}^{k}_{\delta}(\BB^{N}_R)$. Moreover,
		\begin{align*}
			\| \mathbf{N}_{d} (\mathbf{f}) \|_{\mathcal{H}^{k}(\BB^{N}_R)} &\lesssim \delta^{2} \,, \\
			\| \mathbf{N}_{d} (\mathbf{f}) - \mathbf{N}_{d} (\mathbf{g}) \|_{\mathcal{H}^{k}(\BB^{N}_R)} &\lesssim ( \| \mathbf{f} \|_{\mathcal{H}^{k}(\BB^{N}_R)} + \| \mathbf{g} \|_{\mathcal{H}^{k}(\BB^{N}_R)} ) \| \mathbf{f} - \mathbf{g} \|_{\mathcal{H}^{k}(\BB^{N}_R)} \,,
		\end{align*}
		for all $\mathbf{f},\mathbf{g}\in \mathcal{H}^{k}_{\delta}(\BB^{N}_R)$ and all $d\in\overline{\BB^{N}_{R_0^{-1}}(d_0)}$.
	\end{lem}
	\begin{proof}
		By \cite[Lemma 4.2]{ostermann2024stable}, it suffices to verify the existence of a constant $c > 0$ such that  
		\begin{align*}
			\kappa_d(y) \geq c
		\end{align*}
		for all $|y| \leq 1$ and for all $d \in \overline{\mathbb{B}^{N}_{R_0^{-1}}(d_0)}$. This follows from the estimate  
		\begin{align*}
			|d| \leq |d_0| + R_0^{-1} < R^{-1} \leq 1.
		\end{align*}
	\end{proof}
	
	\subsection{Stabilized nonlinear evolution}
	\begin{defn}
		Let $(N,p,k,R)\in\NN\times\RR_{>1}\times \NN\times\RR_{\ge 1}$ with $k>\frac{N}{2}$. Put
		\begin{equation*}
			\omega_{p} \coloneqq \min\left\{1,s_{p} \right\} > 0
		\end{equation*}
		and let $-\omega_{p}<\omega_{0}<0$ be arbitrary but fixed. We define a Banach space $(\mathcal{X}^{k}(\BB^{N}_R),\|\,.\,\|_{\mathcal{X}^{k}(\BB^{N}_R)})$ by
		\begin{align*}
			\mathcal{X}^{k}(\BB^{N}_R) &\coloneqq \Big\{\mathbf{q} \in C\big( [0,\infty),\mathcal{H}^{k}(\BB^{N}_R) \big) \,\Big|\, \|\mathbf{q}(s) \|_{\mathcal{H}^{k}(\BB^{N}_R)} \lesssim \ee^{\omega_{0}s} \text{ all } s\geq 0 \Big\} \,, \\
			\|\mathbf{q} \|_{\mathcal{X}^{k}(\BB^{N}_R)} &\coloneqq \sup_{s\geq 0} \Big( \ee^{-\omega_{0}s}  \|\mathbf{q}(s) \|_{\mathcal{H}^{k}(\BB^{N}_R)} \Big) \,.
		\end{align*}
		We also define for $\delta>0$ a closed ball
		\begin{equation*}
			\mathcal{X}^{k}_{\delta}(\BB^{N}_R) \coloneqq \{\mathbf{q} \in \mathcal{X}^{k}(\BB^{N}_R) \mid \|\mathbf{q} \|_{\mathcal{X}^{k}(\BB^{N}_R)} \leq \delta \} \,.
		\end{equation*}
	\end{defn}
	
	\begin{defn}
		\label{CorrTerm}
		Let $(N,p,k,R)\in\NN\times\RR_{>1}\times \NN\times\RR_{\ge 1}$. Let $\mathbf{P}_{0,d},\mathbf{P}_{1,d} \in \mathcal{L} \left( \mathcal{H}^{k}(\BB^{N}_R) \right)$ be the spectral projections from  \Cref{SpecThmLbeta} and set $\mathbf{P}_{d} \coloneqq \mathbf{P}_{0,d} + \mathbf{P}_{1,d}$. We define a correction term
		\begin{align*}
			\mathbf{C}_{d} : \, &\mathcal{H}^{k}(\BB^{N}_R) \times \mathcal{X}^{k}_{\delta}(\BB^{N}_R) \rightarrow \mathcal{H}^{k}(\BB^{N}_R) \\&
			(\mathbf{f},\mathbf{q}) \mapsto \mathbf{P}_{d}\mathbf{f} + \mathbf{P}_{0,d} \int_{0}^{\infty} \mathbf{N}_{d}(\mathbf{q}(s')) \dd s' + \mathbf{P}_{1,d} \int_{0}^{\infty} \ee^{-s'} \mathbf{N}_{d}(\mathbf{q}(s')) \dd s' \,.
		\end{align*}
	\end{defn}
	Note that the correction term belongs to the unstable subspace $\ran(\mathbf{P}_{d})\subset\mathcal{H}^{k}(\BB^{N}_R)$. Subtracting it from the initial data in the Duhamel formula stabilizes the nonlinear wave evolution.
	\begin{prop}
		\label{NonLinEvoStableSub}
		Let $d_0 \in \BB^N$ and $(N,p,k,R)\in \NN\times\RR_{>1}\times\NN\times\RR_{\ge 1}$ such that $k> \frac{N}{2}$ and $R<|d_0|^{-1}$.   Let $-\omega_{p}<\omega_{0}<0$. There are constants $R_0>R/(1-|d_0|R)$, $0<\delta_{0}<1$, and $C_{0} > 1$, such that for all $0<\delta\leq\delta_{0}$, $C\geq C_{0}$, all $d\in\overline{\BB^{N}_{R_0^{-1}}(d_0)}$ and all $\mathbf{f}\in\mathcal{H}^{k}(\BB^{N}_R)$ with $\| \mathbf{f} \|_{\mathcal{H}^{k}(\BB^{N}_R)} \leq \frac{\delta}{C}$ there is a unique $\mathbf{q}_{d}\in\mathcal{X}^{k}(\BB^{N}_R)$ such that $\|\mathbf{q}_{d} \|_{\mathcal{X}^{k}(\BB^{N}_R)} \leq \delta$ and
		\begin{equation}
			\label{StabilizedSolution}
			\mathbf{q}_{d}(s) = \mathbf{S}_{d}(s) (\mathbf{f} - \mathbf{C}_{d}(\mathbf{f},\mathbf{q}_{d})) + \int_{0}^{s} \mathbf{S}_{d}(s-s')\mathbf{N}_{d}(\mathbf{q}_{d}(s')) \dd s'
		\end{equation}
		for all $s\geq 0$. Moreover, the data-to-solution map
		\begin{equation*}
			\mathcal{H}^{k}_{\delta}(\BB^{N}_R) \rightarrow \mathcal{X}^{k}(\BB^{N}_R) \,, \qquad \mathbf{f} \mapsto\mathbf{q}_{d} \,,
		\end{equation*}
		is Lipschitz continuous.
	\end{prop}
	\begin{proof}
		The proof follows similarly to \cite[Proposition 4.1]{ostermann2024stable}, with $\beta \in \overline{\BB^{N}_{R_0^{-1}}}$ replaced by $d\in\overline{\BB^{N}_{R_0^{-1}}(d_0)}$. 
	\end{proof}
	
	\subsection{Stable flow near the blowup solution}
	The initial data for our abstract Cauchy problem \eqref{abstractCauchy} are introduced as follows.
	\begin{defn} 
		Let $d_0 \in \BB^N$ and $(N,p,k,R)\in\NN\times\RR_{>1}\times \NN\times\RR_{\ge 1}$ such that $k> \frac{N}{2}$ and $R<|d_0|^{-1}$.  Let $R_0>R/(1-|d_0|R)$,  $d\in\overline{\BB^{N}_{R_0^{-1}}(d_0)}$, and $T>0$. We define the operator
		\begin{equation*}
			\mathbf{Q}_{d,T} : C^{\infty}(\RR^{N})^{2} \rightarrow \mathcal{H}^{k}(\BB^{N}_R) \,, \qquad \mathbf{f} \mapsto \mathbf{f}^{T} + \mathbf{f}_{0}^{T} - \mathbf{f}_{d} \,,
		\end{equation*}
		where $\mathbf{f}^{T},\mathbf{f}_{0}^{T},\mathbf{f}_{d}\in C^{\infty}(\overline{\BB^{N}_R})^{2}$ are defined by
		\begin{align*}
			&&
			\mathbf{f}^{T}(y) &=
			\begin{bmatrix}
				\hfill T^{s_{p}} f_{1}(Ty) \\
				T^{s_{p}+1} f_{2}(Ty)
			\end{bmatrix}
			\,, &
			\mathbf{f}_{0}^{T}(y) &=
			\begin{bmatrix}
				T^{s_{p}} \kappa_0 \gamma(d_0)^{-s_p} (1+ T d_0\cdot y)^{-s_p} \\
				T^{s_{p}+1} s_p \kappa_0 \gamma(d_0)^{-s_p} (1+ T d_0\cdot y)^{-s_p-1} 
			\end{bmatrix}
			\,, & \\
			&&
			\mathbf{f}_{d}(y) &=
			\begin{bmatrix}
				\hfill \kappa_0 \gamma(d)^{-s_{p}} (1+d\cdot y)^{-s_{p}} \\
				s_{p} \kappa_0 \gamma(d)^{-s_{p}} (1+d\cdot y)^{-s_{p}-1}
			\end{bmatrix}
			\,. && &
		\end{align*}
	\end{defn}
	To guarantee a stabilized evolution for such data, we need a smallness property of the initial data operator. This is ensured by the following lemma.
	\begin{lem}
		\label{InitDatOpSmall}
			Let $d_0 \in \BB^N$ and $(N,p,k,R)\in\NN\times\RR_{>1}\times \NN\times\RR_{\ge 1}$ such that $k> \frac{N}{2}$ and $R<|d_0|^{-1}$.  Let $R_0>R/(1-|d_0|R)$,  $d\in\overline{\BB^{N}_{R_0^{-1}}(d_0)}$, and $T\in\overline{\BB^{1}_{\frac{1}{2}}(1)}$. We have
		\begin{equation*}
			\mathbf{Q}_{d,T}(\mathbf{f}) = \mathbf{f}^{T} + (T-1) \mathbf{f}_{1,d} + (d^{i}-d_0^i) \mathbf{f}_{0,d,i} + \mathbf{r}(d,T) \,,
		\end{equation*}
		for any $\mathbf{f}\in C^{\infty}(\RR^{N})^{2}$, where $\mathbf{f}_{1,d},\mathbf{f}_{0,d,i}$ are the symmetry modes from \Cref{EV}, and
		\begin{equation*}
			\| \mathbf{r}(d,T) \|_{\mathcal{H}^{k}(\BB^{N}_R)} \lesssim |T-1|^{2} + |d-d_0|^{2}
		\end{equation*}
		for all $d\in\overline{\BB^{N}_{R_0^{-1}}(d_0)}$ and all $T\in\overline{\BB^{1}_{\frac{1}{2}}(1)}$.
	\end{lem}
	\begin{proof}
		The proof follows similarly to  \cite[Lemma 4.3]{ostermann2024stable}, with $\beta \in \overline{\BB^{N}_{R_0^{-1}}}$ replaced by $d\in\overline{\BB^{N}_{R_0^{-1}}(d_0)}$, $\mathbf{U}_{\beta,T}$ replaced by $	\mathbf{Q}_{d,T}$, and $|\beta|$ replaced by $|d-d_0|$. 
	\end{proof}
	
	It remains to remove the correction term in \Cref{StabilizedSolution} to obtain a global solution to \Cref{NonLinDuhamel}.
	\begin{prop}
		\label{StableNonlinFlowMild}
		Let $d_0\in \BB^N$ and $(N,p,k,R)\in\NN\times\RR_{>1}\times \NN\times\RR_{\ge 1}$ such that $k> \frac{N}{2}$ and $R<|d_0|^{-1}$. Let $0 < \varepsilon < \omega_{p}$. There are constants $0<\delta_{\varepsilon}<1$, $C_{\varepsilon} > 1$ such that for all $0<\delta\leq\delta_{\varepsilon}$, $C\geq C_{\varepsilon}$ and for all real-valued $\mathbf{f}\in C^{\infty}(\RR^{N})^{2}$ with $\| \mathbf{f} \|_{\mathcal{H}^{k}(\RR^{N})} \leq \frac{\delta}{C^{2}}$ there are parameters $d^{*}\in\overline{\BB^{N}_{\frac{\delta}{C}}(d_0)}$, $T^{*}\in\overline{\BB^{1}_{\frac{\delta}{C}}(1)}$ and a unique $\mathbf{q}_{d^{*},T^{*}}\in C\big([0,\infty),\mathcal{H}^{k}(\BB^{N}_R)\big)$ such that $\|\mathbf{q}_{d^{*},T^{*}}(s) \|_{\mathcal{H}^{k}(\BB^{N}_R)} \leq \delta \ee^{(-\omega_{p}+\varepsilon)s}$ and
		\begin{equation}
			\label{CauchyMild}
			\mathbf{q}_{d^{*},T^{*}}(s) = \mathbf{S}_{d^{*}}(s)\mathbf{Q}_{d^{*},T^{*}}(\mathbf{f}) + \int_{0}^{s} \mathbf{S}_{d^{*}}(s-s')\mathbf{N}_{d^{*}}(\mathbf{q}_{d^{*},T^{*}}(s')) \dd s'
		\end{equation}
		for all $s\geq 0$.
	\end{prop}
	\begin{proof}
			The proof follows similarly to  \cite[Proposition 4.2]{ostermann2024stable}, with $\beta^* \in \overline{\BB^{N}_{\frac{\delta}{C}}}$ replaced by $d^{*}\in\overline{\BB^{N}_{\frac{\delta}{C}}(d_0)}$.
	\end{proof}
	
	By Sobolev embedding and induction, the just obtained mild solution is in fact a jointly smooth classical solution, which gives the proof of \Cref{StableNonlinFlowClassic}.

\bibliographystyle{abbrv}
\bibliography{ref-nlw-superconformal}

\end{document}